\documentclass[12pt]{article}

\usepackage{graphicx}
\usepackage{epsfig}
\usepackage{psfrag}
\usepackage{wrapfig}
\usepackage[all]{xy}

\usepackage{fullpage}
\usepackage{setspace}
\usepackage{mathptmx}
\usepackage{url}
\usepackage{booktabs}

\usepackage{footmisc}
\newcommand\wordcount{
    \immediate\write18{texcount -sum -1 \jobname.tex > 'count.txt'}
\input{count.txt}words}

\usepackage{natbib}

\usepackage{amsmath}
\usepackage{amsfonts}
\usepackage{amsthm}
\usepackage{bbm}
\usepackage{array}

\newtheorem{prop}{Proposition}

\newtheorem{definition}{Definition}
\newtheorem{assumption}{Assumption}

\def\E{{\rm E}\,}

\def\Supp{{\rm Supp}\,}

\begin{document}

\sloppy

\title{Local average causal effects and superefficiency}
\author{Peter M. Aronow\thanks{Peter M. Aronow is Assistant Professor, Departments of Political Science and Biostatistics, Yale University, 77 Prospect St., New Haven, CT 06520 (Email: peter.aronow@yale.edu). The author thanks Don Green, Cyrus Samii, Jas Sekhon  and Mark van der Laan for helpful comments. All remaining errors are the author's responsibility.}} 

\maketitle

\begin{abstract}
Recent approaches in causal inference have proposed estimating average causal effects that are local to some subpopulation, often for reasons of efficiency. These inferential targets are sometimes data-adaptive, in that they are dependent on the empirical distribution of the data. In this short note, we show that if researchers are willing to adapt the inferential target on the basis of efficiency, then extraordinary gains in precision can be obtained. Specifically, when causal effects are heterogeneous, any asymptotically normal and root-$n$ consistent estimator of the population average causal effect is superefficient for a data-adaptive local average causal effect. Our result illustrates the fundamental gain in statistical certainty afforded by indifference about the inferential target.
\end{abstract}

\section{Introduction}


When causal effects are heterogeneous, then inferences depend on the population for which causal effects are estimated. Although population average causal effects have traditionally been the inferential targets, recent results have focused on estimating average causal effects that are {\it local} to some subpopulation for reasons of efficiency. These approaches include trimming observations based on the distribution of the propensity score \citep{crumpetal}, using regression adjustment to estimate reweighted causal effects \citep{angristpischke, humphreys2009}, or implementing calipers for propensity-score matching \citep{rosenbaumrubin1985,austin2011}. In some cases, the target parameter is dependent on the empirical distribution of the data, including cases where the researcher is explicitly conducting inference on, e.g., the average treatment effect among the treated conditional on the observed covariate distribution \citep{abadieimbens2002}, or other causal sample functionals \citep{aronowetal,balzeretal}, without revision to the estimator being used.


This approach, taken at full generality, implies a form of indifference to which population causal effects are measured for. This indifference can be codified in the form of a {\it data-adaptive} target parameter \citep{vanderlaanetal} that is allowed to vary with the data depending on which subpopulation's local average causal effect is best estimated. When treatment effects are heterogeneous, adaptively changing the target parameter on the basis of efficiency yields an unusual result: if the population average causal effect can be consistently estimated with a root-$n$ consistent and asymptotically normal estimator $\hat\theta$, then the same estimator $\hat\theta$ is always superefficient (i.e., faster than root-$n$ consistent) for a data-adaptive local average causal effect. Furthermore, with an additional regularity condition on mean square convergence, we show that the mean square error of $\hat \theta$ for a data-adaptive local average causal effect is of $o(n^{-{1}})$.


\section{Results}

Consider a full data probability distribution $G$ with an associated causal effect distribution $\tau$ with finite expectation $\E_G[\tau]$, where $\E_G[.]$ denotes the expectation over the distribution $G$. We impose a regularity condition on $\tau$ establishing non-degeneracy of $\tau$.


\begin{assumption}[Effect heterogeneity]
$\min\left(\sup \left( \Supp(\tau) \right) - \E_G[\tau], \E_G[\tau] - \inf \left(\Supp(\tau)\right)\right) = c > 0.$
\end{assumption}

We observe an empirical distribution $F_n$. Suppose we have an root-$n$ consistent and asymptotically normal estimator of the average causal effect $\E_G[\tau]$, $\hat \theta$.
\begin{definition}
An estimator $\hat\theta$ is root-$n$ consistent and asymptotically normal for $\theta_0$ if $\sqrt n (\hat\theta - \theta_0) = \mathcal{N}(0,\sigma^2) + o_p(1)$, for some $0 < \sigma^2 < \infty$.
\end{definition}
We now define the target parameter, $\theta_{F_n}$.
\begin{definition}
Let the target parameter $$\theta_{F_n} =  \left\{
     \begin{array}{lr}
      \hat \theta & : | \hat \theta - \E_G[\tau] | \leq c  \\
     \E_G[\tau] + c & : \hat\theta - \E_G[\tau] > c  \\
     \E_G[\tau] - c  & : \hat\theta - \E_G[\tau] < -c 
     \end{array}
   \right.,$$
   where, as in Assumption 1, $c = \min\left(\sup \left( \Supp(\tau) \right) - \E_G[\tau], \E_G[\tau] - \inf \left(\Supp(\tau)\right)\right)$.
\end{definition}

The target parameter adapts naturally to the closest value in an interval surrounding $\E_G[\tau]$, where the width of the interval is defined by the support of $\tau$. We formalize how each $\theta_{F_n}$ is a local average treatment effect.

\begin{prop}
There exists a nonnegative weighting associated with each empirical distribution $F_n$, $w_{F_n}$, such that across all $F_n$, $\theta_{F_n} = \frac{\E_G[w_{F_n} \tau]}{\E_G[w_{F_n}]}$.
\end{prop}
A proof of Proposition 1 follows directly from the fact that a weighted mean can obtain any value in the interval defined by the infimum and supremum of its distribution's support. We now prove the the superefficiency of $\hat\theta$.




\begin{prop}
Suppose that Assumption 1 holds. Then for any root-$n$ consistent and asymptotically normal estimator of $\E_G[\tau]$, $\hat\theta$, $\sqrt n(\hat\theta - \theta_{F_n}) = o_p(1).$ 
\end{prop}
\begin{proof}
The author thanks Jas Sekhon for suggesting the following proof strategy. Decompose $\hat\theta$ into $\tilde\theta=\mathcal{N}(\E_G[\tau],\sigma^2/n)$ and $u = o_p(n^{-1/2})$, so that $\hat\theta = \tilde\theta + u$. Since $(\tilde\theta -  \theta_{F_n})$ is $o_p(a_n)$ for any positive sequence $(a_n)$, the rate of convergence of $\hat \theta$ is at worst governed by the bound ensured by $u$'s $o_p(n^{-1/2})$ convergence. To prove the claim, note that for any positive $\epsilon$, $
\Pr\left( \frac{| \tilde\theta -  \theta_{F_n} |}{a_n} \geq \epsilon \right) \leq \Pr\left( \tilde\theta - \theta_{F_n} \neq 0 \right) = 2\Phi(-c \sqrt{n}/\sigma).
$
Since $\lim_{n\rightarrow \infty}  2\Phi(-c \sqrt{n}/\sigma) = 0$, $(\tilde\theta -  \theta_{F_n})$ is $o_p(a_n)$. Thus $\hat\theta - \theta_{F_n} =  o_p(a_n) +  o_p(n^{-1/2}) = o_p(n^{-1/2})$, yielding the result.
%
\end{proof}

When an additional regularity condition is imposed on the convergence of $\hat\theta$ to normality, a stronger result can be obtained about the rate of mean square convergence.
\begin{prop}
Suppose that $\hat\theta$ obeys $\sqrt n (\hat\theta - \E_G[\tau]) = \mathcal{N}(0,\sigma^2) + \epsilon$, where $\E_G[\epsilon^2] = o(n^{-1/2}).$  Then $\E_G[(\hat\theta - \theta_{F_n})^2]  = o(n^{-1}).$ 
\end{prop}
\begin{proof}
We will show that the mean square error of $(\tilde\theta -  \theta_{F_n})$ converges to zero sufficiently quickly, implying that the rate of convergence of $\hat \theta$ is at worst governed by the mean square error bound ensured by $\epsilon$'s convergence rate. To obtain the rate of convergence of the mean square error of $\tilde\theta$, we integrate over its squared deviation from the target parameter.  Within $c$ of $\E_G[\tau]$, the squared deviation is zero, thus we need only integrate over the squared deviation over the tails of the normal distribution. To ease calculations, we obtain an upper bound by integrating over the squared deviation from $\E_G[\tau]$, rather than from $\theta_{F_n}$:
\begin{align*}
\E_G[(\tilde\theta - \theta_{F_n})^2] & \leq 2 \int_{c}^{\infty} x^2 \frac{\sqrt{n}}{\sigma \sqrt{2\pi}}e^{-\frac{x^2 n}{2 \sigma^2}} \\
& = c \sigma \sqrt{\frac{2}{\pi} } \frac{ e^{-\frac{c^2 n}{2 \sigma^2}}}{ \sqrt{n}}+2 \sigma^2\frac{\Phi \left(\frac{-c \sqrt{n}}{\sigma}\right)}{n} \\
& = o(n^{-1}).
\end{align*}
Since $\E_G[(\tilde\theta - \theta_{F_n})^2] = o(n^{-1})$ and $n^{-1/2} \E_G[\epsilon^2] = o(n^{-1})$, the Cauchy-Schwarz inequality ensures that $\E_G[(\hat\theta - \theta_{F_n})^2]  = o(n^{-1}) + o(n^{-1}) = o(n^{-1})$.
\end{proof}



%

%

\section{Discussion}

Our results highlight the additional certainty obtained by indifference about the population for which average causal effects are measured. It is well known that efficiency gains may be obtained through data-adaptive inference. But the extent to which the researcher benefits from indifference about the target parameter has been understated. Under treatment effect heterogeneity -- a precondition for locality to be a concern -- all root-$n$ consistent and asymptotically normal estimators of the average treatment effect are superefficient for a local average treatment effect. And while we do not speak to the substantive implications of indifference in scientific inquiry, we show how such indifference yields greatly increased statistical certainty.


\begingroup


\begin{thebibliography}{}
\bibitem[\protect\citeauthoryear{Abadie and Imbens}{2002}]{abadieimbens2002}  Abadie, A. and Imbens, G. 2002. Simple and bias-corrected matching estimators for average treatment effects. NBER technical working paper no. 283.

\bibitem[\protect\citeauthoryear{Angrist and Pischke}{2009}]{angristpischke} Angrist, J.D. and Pischke, J.S. 2009. {\it Mostly harmless econometrics: An empiricist's companion}. Princeton, NJ: Princeton University Press. 

\bibitem[\protect\citeauthoryear{Aronow, Green, and Lee}{2014}]{aronowetal}  Aronow, P.M., Green, D.P. and Lee, D.K.K. 2014. Sharp bounds on the variance in randomized experiments. {\it Annals of Statistics}. 42(3) 850--871.
 

\bibitem[\protect\citeauthoryear{Austin}{2011}]{austin2011}  Austin, P.C., 2011. Optimal caliper widths for propensity-score matching when estimating differences in means and differences in proportions in observational studies. {\it Pharmaceutical Statistics}, 10(2), pp.150--161.

\bibitem[\protect\citeauthoryear{Balzer, Petersen, and van der Laan}{2015}]{balzeretal}  Balzer, L.B., Petersen, M.L. and van der Laan, M.J. 2015. Targeted estimation and inference for the sample average treatment effect. bepress.

\bibitem[\protect\citeauthoryear{Crump et al.}{2009}]{crumpetal}   Crump, R.K., Hotz, V.J., Imbens, G.W. and Mitnik, O.A. 2009. Dealing with limited overlap in estimation of average treatment effects. {\it Biometrika}.

\bibitem[\protect\citeauthoryear{Humphreys}{2009}]{humphreys2009}  Humphreys, M., 2009. Bounds on least squares estimates of causal effects in the presence of heterogeneous assignment probabilities. Manuscript, Columbia University.

\bibitem[\protect\citeauthoryear{Rosenbaum and Rubin}{1985}]{rosenbaumrubin1985}  Rosenbaum P.R. and Rubin D.B. 1985. Constructing a control group using multivariate matched sampling methods that incorporate the propensity score. {\it American Statistician} 39(1):33--38

\bibitem[\protect\citeauthoryear{van der Laan, Hubbard, and Pajouh}{2015}]{vanderlaanetal}  van der Laan, M.J., Hubbard, A.E. and Pajouh, S.K. (2013).  Statistical inference for data adaptive target parameters. bepress.

\end{thebibliography}
\end{document}